\documentclass[10pt]{amsart}

\usepackage{
amsfonts,
latexsym,
amssymb,
amsmath,
amsthm,
enumerate,
verbatim,
mathrsfs,
color,
}

\usepackage{url}

\newcommand{\labbel}[1]{\label{#1} [[{\bf #1}]]}  %
\renewcommand{\labbel}{\label}

 \definecolor{reed}{RGB}{0,0,100}

\newcommand{\red}[1]{{#1}}

\newtheorem{theorem}{Theorem}[section]
\newtheorem{lemma}[theorem]{Lemma}

\newtheorem{proposition}[theorem]{Proposition} 
 
\newtheorem{corollary}[theorem]{Corollary}

\newtheorem*{theorem*}{Theorem}
\newtheorem*{corollary*}{Corollary}

\theoremstyle{definition}
\newtheorem{definition}[theorem]{Definition}

\newtheorem{problem}[theorem]{Problem}

\theoremstyle{remark}
\newtheorem{remark}[theorem]{Remark}

\numberwithin{equation}{section}

\newcommand{\hyphe}{\frac{\ }{\ }}

 \allowdisplaybreaks[0]
\begin{document}

\title[An $n \frac{1}{2} $-ary near-unanimity term]
{Between an $n$-ary and an $n{+}1$-ary\\near-unanimity term}

\author{Paolo Lipparini} 
\address{Dipartimento 
 di Matematica\\Viale della  Ricerca  Intermedia
\\Universit\`a di Roma ``Tor Vergata'' 
\\I-00133 ROME ITALY}
\urladdr{http://www.mat.uniroma2.it/\textasciitilde lipparin}


\thanks{Work performed under the auspices of G.N.S.A.G.A. Work 
partially supported by PRIN 2012 ``Logica, Modelli e Insiemi''.
The author acknowledges the MIUR Department Project awarded to the
Department of Mathematics, University of Rome Tor Vergata, CUP
E83C18000100006.}

\begin{abstract}
We devise a condition strictly between the existence of 
an $n$-ary and an $n{+}1$-ary near-unanimity term.
We evaluate exactly the distributivity and modularity levels 
implied by such a condition.
\end{abstract}

\maketitle

\section{Introduction} \labbel{intro}

Varieties with  a near-unanimity term 
form a distinguished class of congruence distributive varieties
and have been studied from the 70's in the past century
\cite{BP,Mi}.
More recently, a fundamental paper by
Berman,  Idziak,  Markovi\'c,  McKenzie, Valeriote, Willard 
\cite{BIM} 
showed that, among congruence distributive 
varieties,  near-unanimity terms play a very important role
in tractability problems \cite{IMM}. 
Recent results about near-unanimity terms
include 
\red{\cite{Bar,CCV,Dal,DHM,Mar,MZ,S,Z}.}

A majority term is the simplest
ternary case of a near-unanimity term
and corresponds exactly to $2$-distributivity,
the first and strongest nontrivial level in the    
J{\'o}nsson hierarchy. 
Levels in this sense \cite{FV}  are measured 
by the minimal number of J{\'o}nsson terms
witnessing congruence distributivity, as recalled in Theorem \ref{cdd}(2)
below.
For near-unanimity terms of larger arity
the conditions overlap no more
and near-unanimity provides a condition
strictly stronger than congruence distributivity
\cite{Mi}.
For $n \geq 3$, the existence of an 
 $n$-ary near-unanimity term,
for short, an \emph{$n$-near-unanimity term},  implies
$2n{-}4$-distributivity
\cite[Theorem 2]{Mi}.  
In \cite[Theorem 3.6]{mis} we showed that the above result
is sharp, even when restricted to  locally finite varieties 
with a symmetric near-unanimity term.

Since, by the above results,
the existence of an 
 $n{+}1$-near-unanimity term implies
$2n{-}2$-distributivity,
the results suggest that there
might possibly  be a condition   strictly between  
an 
 $n$-
and an 
 $n{+}1$-near-unanimity term
and which
implies 
$2n{-}3$-distributivity.
Henceforth, an
 ``$n\frac{1}{2} $-near-unanimity term'' 
is a suitable name for such a condition.
A candidate for such a condition
has been proposed
in \cite[Definition 4.7]{mis},
involving an $n{+}2$-ary 
term.
In Section \ref{bsec}  we show
that 
the condition proposed in \cite{mis} does
satisfy the required properties, hence actually deserves the name 
of an \emph{$n\frac{1}{2} $-near-unanimity term}.
More involved arguments 
in Section \ref{dl} 
show that the result about the distributivity
level is optimal, namely, that 
an $n\frac{1}{2} $-near-unanimity term
  does not necessarily imply
$2n{-}4$-distributivity.
Corresponding results are proved for modularity levels
in Section \ref{mod}.

Under the above terminological conventions,
a compact way to express the main results of the present note 
goes as follows,
where parameters vary on integers and half-integers.

\begin{theorem} \labbel{tut}
Suppose that $h,k \in \mathbb N \cup (\mathbb N + \frac{1}{2}) $,
$h,k \geq 3$.  
  \begin{enumerate}   
 \item  
Suppose that  $h <k$. Then every variety with an $h$-near-unanimity term   
has a $k$-near-unanimity term. Moreover, there is 
a locally finite variety    with a $k$-near-unanimity term   
without an $h$-near-unanimity term.
\item
Every variety with an $h$-near-unanimity term   
is $2h{-}4$-distributive. There is 
a locally finite variety    with an $h$-near-unanimity term   
which is not $2h{-}5$-distributive.
\item
Every variety with an $h$-near-unanimity term   
is $2h{-}3$-modular. If $h \geq 4$ there is 
a locally finite variety    with an $h$-near-unanimity term   
which is not $2h{-}4$-modular.
 \end{enumerate} 
 \end{theorem} 

Details for the proof of Theorem \ref{tut}(1) shall be given 
at the end of Section \ref{bsec}. 
The remaining items shall be proved
in Section \ref{mod}.

A few variations, still 
between  
an 
 $n$-
and an 
 $n{+}1$-near-unanimity term,
are presented in Section \ref{var}. 
Section \ref{prob} presents some problems.

\section{Preliminaries} \labbel{prel}

Throughout, $n$ is a natural number $ \geq 2$,
frequently, $\geq 3$.  

A
\emph{near-unanimity term}  
is a term $u$ of arity $\geq 3$ and  such that    all the equations of the form
\begin{equation*} 
u(x,x, \dots, x,y,x, \dots, x,  x) =x
\end{equation*}    
are satisfied (in some given algebra or variety),
 with just one occurrence of $y$
in any possible position. 
For notational convenience, an $n$-ary near-unanimity 
term shall be sometimes called 
an  \emph{$n$-near-unanimity term}. 

An $n$-ary term $u$ is 
\emph{symmetric} if  the equations 
\begin{equation*}
    u(x_1, \dots, x_n)=
 u(x_{ \tau  (1)}, \dots, x_ { \tau  (n)})
  \end{equation*}
hold, for all permutations $ \tau $ of
$\{ 1, \dots, n\}$. 

The following theorem 
provides  important  Maltsev conditions
characterizing congruence distributivity.

\begin{theorem} \labbel{cdd}
For every variety $\mathcal V$, the following conditions are equivalent. \begin{enumerate} 
  \item   
$\mathcal V$ is congruence distributive.
\item
(J{\'o}nsson \cite{JD}) For some natural number $k$,
$\mathcal V$ has a sequence of \emph{J{\'o}nsson terms},
that is,  terms 
$t_0, \dots, t_k$ satisfying 
\begin{align} \labbel{j1} 
  x&= t_0(x,y,z),  &  \\
\labbel{j2}
  x&=t_i(x,y, x), 
\quad \ \ \  \text{ for } 
0 \leq i \leq k,
 \\ 
 \labbel{j3}  
\begin{split}  
 t_{i}(x,x,z) &=
t_{i+1}(x,x,z),
 \quad \text{ for even $i$,\ } 
0 \leq i < k,
  \\ 
 t_{i}(x,z,z)&=
t_{i+1}(x,z,z),
  \quad \text{ for odd $i$,\ } 
0 \leq i < k,
 \end{split}
 \\ 
\labbel{j4}
t_{k}(x,y,z)&=z.  &
\end{align}   
\item
(Kazda, Kozik,  McKenzie,  Moore \cite{KKMM}) 
For some natural number $k$,
$\mathcal V$ has a sequence of
  \emph{directed J{\'o}nsson terms},
that is, terms satisfying 
\eqref{j1}, \eqref{j2}, \eqref{j4} and   
\begin{equation}\labbel{D}
  t_{i}(x,z,z)=
t_{i+1}(x,x,z)  \quad \text{ for \ } 
0 \leq i < k.
  \end{equation}    
\end{enumerate}  
 \end{theorem} 

To the best of our knowledge,
 directed J{\'o}nsson terms first appeared 
\red{implicitly in the proof of 
\cite[Theorem 2.3]{Mck} and explicitly
(but unnamed)
in \cite[Theorem 4.1]{Z}.} 

\begin{definition} \labbel{ndist}    
A variety $\mathcal V$ is said to be 
\emph{$k$-distributive} if $\mathcal V$
has a sequence  
$t_0, \dots, t_k$ of J{\'o}nsson terms.
It is standard to see that a variety $\mathcal V$ 
is $k$-distributive if and only if
$\mathcal V$ satisfies the  congruence identity 
$ \alpha ( \beta \circ \gamma )
 \subseteq  
\alpha \beta \circ \alpha \gamma \circ {\stackrel{k}{\dots}}$. 
In this and similar identities 
juxtaposition
denotes intersection
and 
$ {\stackrel{k}{\dots}}$ 
means that we are considering 
 $k$ factors, namely, $k-1$ occurrences of  $ \circ $.
If, say, $k$ is even, then we might  write
$\alpha \beta \circ \alpha \gamma \circ
 {\stackrel{k}{\dots}} \circ \alpha \gamma $
when we want to make clear that $ \alpha \gamma $
 is the last factor. \red{Notice that an inclusion of the form
$A \subseteq B$ is equivalent to the identity $AB=A$,
hence we can always use the expression ``identity''.} 
\end{definition}

\begin{theorem} \labbel{mit}
(Mitschke \cite[Theorem 2]{Mi}) A variety 
with a near-unanimity term is congruence distributive.

In more detail, for $n \geq 3$,  a variety with an
$n$-ary near-unanimity term is $2n{-}4$-distributive  \cite{Mi}
and has a sequence $t_0, \dots, t_{n-1}$
of directed J{\'o}nsson terms \cite[Section 5.3.1]{BK}.
 \end{theorem} 

\red{Notice that the counting conventions
about the number of directed J{\'o}nsson terms are
not uniform through the literature
(not even in the works by the present author).} 

A direct proof 
(not relying on Theorem \ref{cdd})
that a variety 
with a near-unanimity term is congruence distributive
can be found in  
 \cite[Lemma 1.2.12]{KP}.
The proof is credited to 
E. Fried.
Compare also \cite[Section 5]{B}.

We have showed in \cite[Theorem 3.6 and Remark 4.5]{mis}
that  no parts of the second statement
in Theorem \ref{mit} can be improved.

\begin{theorem} \labbel{daythm}
\cite{D} For every variety $\mathcal V$, the following conditions are equivalent. \begin{enumerate}   
\item
$\mathcal V$ is congruence modular.
\item
For some natural number $k$,
  $\mathcal V$ has a  sequence 
of \emph{Day terms},
namely, $4$-ary terms $t_0, \dots, t_k$ satisfying 
\begin{align*}         
x &=t_i(x,y,y,x),  \quad \quad  \ \,    \text{ for  } 0 \leq i \leq k,  
\\
 x&=t_0(x,y,z,w), 
\\    
\begin{split}     
 t_i(x,x,w,w)&=t_{i+1}(x,x,w,w), \quad \text{ for $i$  even, }  0 \leq i < k, 
\\
 t_i(x,y,y,w)&=t_{i+1}(x,y,y,w), \quad \; \text{ for $i$  odd, }  0 \leq i <  k,
\end{split}
\\ 
 t_{k}(x,y,z,w)&=w.
\end{align*}
\end{enumerate}   
 \end{theorem} 

A congruence modular variety $\mathcal V$ is \emph{$k$-modular} 
if $\mathcal V$ 
has a sequence  
$t_0, \dots, t_k$ of Day terms; this is equivalent
to the congruence identity 
$ \alpha ( \beta \circ \alpha \gamma  \circ \beta )
 \subseteq  
\alpha \beta \circ \alpha \gamma \circ {\stackrel{k}{\dots}}$.

\section{Between an $n$-ary and an $n{+}1$-ary 
near-unanimity term} \labbel{bsec} 

\begin{definition} \labbel{12}
\cite[Definition 4.7]{mis} 
If $n \geq 2$, an \emph{$n\frac{1}{2} $-near-unanimity term}
is an $n{+}2$-ary term $u $ such that the following equations hold.
\begin{align} 
\labbel{b1}   
&u(z,z, x, x, \dots, x) = x,
\\
\labbel{b2}
&u(x, \dots, x,\underset{i}{z},x, \dots, x) = x, && \text{for $2 \leq i \leq n+2$,} 
\\ \labbel{b3}
&u(x, x, x, z, z, \dots, z) = u(x, z, z, z, z, \dots, z).
 \end{align}   
 \end{definition}

As we mentioned in the introduction,
the terminology comes from the fact that 
 an $n\frac{1}{2} $-near-unanimity term
is a notion strictly between 
an $n$-near-unanimity term
and an $n+1$-near-unanimity term,
as we will show in Theorem \ref{ip}.
In order to simplify some parts in the proof of  
Theorem \ref{ip} we need the following
proposition of independent interest.

\begin{proposition} \labbel{ipp}
If $n \geq 3$, then every variety with an
$n\frac{1}{2} $-near-unanimity term
 is $2n{-}3$-distributive. 
 \end{proposition}  

\begin{proof}
If $u$ is an $n\frac{1}{2} $-near-unanimity term,  define
\begin{align*}
t_0(x,y,z) &= x,
\\
t_1(x,y,z) &= u(x,x, x, x,x, x, \dots,x, x,x, y, z),
\\
t_2(x,y,z) &= u(x,x, x, x, x, x,\dots, x, x,x, z, z),
\\
t_3(x,y,z) &= u(x,x, x, x, x, x,\dots, x, x,y, z, z),
\\
t_4(x,y,z) &= u(x,x, x, x, x, x,\dots, x, x,z, z, z),
\\
t_5(x,y,z) &= u(x,x, x, x,x, x, \dots, x, y,z, z, z),
\\
&\dots
\\
t_{2n-7}(x,y,z) &= u(x,x, x, x,y, z, \dots, z, z,z, z, z), 
\\
t_{2n-6}(x,y,z) &= u(x,x, x, x,z, z, \dots, z, z,z, z, z), 
\\
t_{2n-5}(x,y,z) &= u(x,x, x, y,z, z, \dots, z, z,z, z, z), 
\\
t_{2n-4}(x,y,z) &= u(x,y, z, z,z, z, \dots, z, z,z, z, z), 
\\
t_{2n-3}(x,y,z) &= z . \qedhere
 \end{align*} 
 \end{proof} 

\begin{remark} \labbel{mki}
(a) Notice that the case $i=3$
in equation \eqref{b2} has not been used in the proof
of Proposition \ref{ipp}.  

(b) With the only exception of the bottom lines,
the proof of  Proposition \ref{ipp}
implicitly uses directed J{\'o}nsson  terms.
Compare \cite[Section 5.3.1]{BK} 
and \cite[Observation 1.2]{KKMM}.   

In fact, the terms $t_0, t_1, t_3, t_5, \dots, t _{2n-7},t _{2n-5}  $ 
above satisfy the equations \eqref{j1}, \eqref{j2} and  \eqref{D}  for 
directed J{\'o}nsson terms. At the end, the relations change,
we have $ t _{2n-5}(x,z,z) =  t _{2n-4}(x,z,z)$
and $ t _{2n-4}(x,x,z)=z$, instead.
Thus, in the terminology from 
\cite{KKMM}, the sequence  
$ t_1, t_3, t_5, \dots, t _{2n-7},t _{2n-5},  t _{2n-4}  $
is a sequence of \emph{directed Gumm terms}. 
Notice that we do not need the case $i=2$
in equation \eqref{b2} in order to get 
a sequence of  directed Gumm terms.

Directed Gumm terms characterize congruence modularity,
they do not necessarily imply congruence distributivity.
However, using the case $i=2$
in  \eqref{b2}, we have in addition 
$ t _{2n-4}(x,y,x)=x$ and this further equation is enough to get
congruence distributivity.

(c) In other words, 
relabeling the terms, while a sequence 
$d_1, d_2, \dots, d_{m-2}, d_{m-1}$ of directed J{\'o}nsson terms
implies $2m{-}2$-distributivity, \red{under the counting
convention from  \cite[Observation 1.2]{KKMM}, }
 a sequence $d_1, d_2, \dots, d_{m-2}, q $ of directed Gumm terms
implies $2m{-}3$-distributivity, \emph{provided the term $q$
satisfies the additional equation $q(x,y,x)=x$.} 

General forms of similar equations have been 
studied in \cite{KV}.
 \end{remark}    

Lattice operations shall be denoted 
by $+$ and juxtaposition. Complement in Boolean algebras 
is denoted by  $'$.

\begin{theorem} \labbel{ip}
Let $n \geq 3$.
  \begin{enumerate}   
\item 
Every variety  with an
$n\frac{1}{2} $-near-unanimity term has
an $n{+}1$-ary near-unanimity term. 
\item 
There is a locally finite variety  
with an $n{+}1$-ary near-unanimity term 
but without an
$n\frac{1}{2} $-near-unanimity term.
\item
Every variety  with an $n$-ary near-unanimity term
has an $n\frac{1}{2} $-near-unanimity term.
\item
There is a locally finite variety 
with an 
$n\frac{1}{2} $-near-unanimity term
but without an $n$-ary near-unanimity term. 
  \end{enumerate}  
 \end{theorem}

\begin{proof}
(1) If $u$ is an  
$n\frac{1}{2} $-near-unanimity term,
then the $n{+}1$-ary term $v$ defined by 
\begin{equation*}\labbel{vvv}
v(x_1, x_2, x_3, \dots , x_{n+1})=
 u(x_1, x_1, x_2, x_3, \dots , x_{n+1})
 \end{equation*}    
is a near-unanimity term, by \eqref{b1} and \eqref{b2}.

 Notice that we have not used \eqref{b3},  and we do not need
the case $i=2$ in \eqref{b2}.   
Notice further that the present argument
works also in the case $n=2$.
This aspect shall be put in a clearer light in  
Remark \ref{dueemm} below.

(2)
In \cite[Definition 3.5]{mis} 
we constructed a locally finite  variety $\mathcal N_{n+1}$
and we showed in \cite[Theorem 3.6]{mis}
that $\mathcal N_{n+1}$ has  an $n{+}1$-ary near-unanimity term  
but is not $2n{-}3$-distributive.
By Proposition \ref{ipp}, $\mathcal N_{n+1}$ witnesses (2).

(3) If $w$ is an $n$-ary near-unanimity term,
then, by adding two initial dummy variables,
$w$ becomes    an
$n\frac{1}{2} $-near-unanimity term.

(4) 
Let $\mathcal V$ be the term reduct of the variety of Boolean algebras
obtained considering  the term 
\begin{equation} \labbel{term}   
u(x_1,  x_2, x_3, \dots , x_{n+2}) =  
  (x_1 + x_2') \prod _{\substack{1 \leq i< j \leq n+2 \\ i \neq 2, j \neq 2}} (x_i + x_j) .   
 \end{equation}

\red{Since $n \geq 3$,} 
  the term $u$
satisfies \eqref{b1} - \eqref{b3}
in Boolean algebras, thus 
$u$, as an operation,
satisfies \eqref{b1} - \eqref{b3}  in $\mathcal V$.
Hence $\mathcal V$ has 
an $n\frac{1}{2} $-near-unanimity term.
$\mathcal V$ is locally finite, being a term reduct of the locally finite
variety of Boolean algebras.

Let $\mathbf 2$
denote the two-element Boolean algebra
with base set $\{ 0,1\}$.
Let $\mathbf A$
be the $u$ term-reduct of the $n$th power of $\mathbf 2$.
We shall show that 
  $B=A \setminus (1,1, \dots , 1)$
is the universe for an algebra in $\mathcal V$.
Indeed, let $b_1, b_2, \dots, b_{n+2} \in B$,
thus, for every $i \leq n+2$, at least one component of $b_i$
is $0$. Actually, we shall not use the assumption that
$b_2 \in B$.       
Since we are working in a power
with $n$ components and
the sequence $b_1, b_3, \dots, b_{n+2} $
has length $n+1$,  
 there are 
 $i \neq j \leq n+2 $ with $  i \neq 2 $, $  j \neq 2$
such that $b_i$ and  $b_j$
have some $0$ at the same component. 
 Thus 
$b=u(b_1, b_2, \dots, b_{n+2})$
has $0$ at that component, hence   
$b \in B$.

Now it is standard to see that 
$\mathcal V$ has not an $n$-ary near-unanimity term.
If, by contradiction,
 $v$ is such a term, then in $\mathbf A$ 
\begin{align*}
&v((0,1,1, \dots, 1), (1,0,1, \dots, 1), \dots, (1,1,1, \dots, 0))=
\\
&(v(0,1,1, \dots, 1), v(1,0,1, \dots, 1), \dots, v(1,1,1, \dots, 0))=
(1,1,1,\dots,1),
  \end{align*}    
contradicting the  just-proved fact that $\mathbf  B$ is a 
subalgebra of $\mathbf A$.
\end{proof}

\begin{problem} \labbel{dimdir}   
Is there a more direct 
proof of item (2)
in Theorem  \ref{ip}
 which does not rely on \cite{mis}? 
For example, is  some variation on
the arguments in the proof
of \ref{ip}(4) enough?
\end{problem}

We are now in the position to give a proof
for item (1) in Theorem \ref{tut}.
Parts (2) and (3) shall need more effort. 

 \begin{proof}[Proof of Theorem \ref{tut}(1)]
 Suppose that $h,k \in \mathbb N \cup (\mathbb N + \frac{1}{2}) $
and
$h,k \geq 3$.  
If $h$ is a half-integer   
and $\mathcal V$ has an
$h$-near-unanimity term, 
then 
$\mathcal V$ has an
$h{+}\frac{1}{2}$-near-unanimity term
by Theorem \ref{ip}(1). 
If $h$ is an integer   
and $\mathcal V$ has an
$h$-near-unanimity term, 
then 
$\mathcal V$ has an
$h{+}\frac{1}{2}$-near-unanimity term
by Theorem \ref{ip}(3). 

 By induction, we get that 
if $h  \leq k$, then every variety with an $h$-near-unanimity term   
has a $k$-near-unanimity term.
It then follows from Theorem \ref{ip}(2)(4) 
that if $h <k$,
then there is 
a locally finite variety    with a $k$-near-unanimity term   
and  without a $k{-} \frac{1}{2} $-near-unanimity term, 
hence without an $h$-near-unanimity term, 
\red{since $h \leq k-\frac{1}{2}$.} 
 \end{proof}

\section{Distributivity levels} \labbel{dl} 

In this section we hint to a proof that
Proposition \ref{ipp} cannot be improved, namely, 
for every  $n \geq 3$, there is a variety with an
$n\frac{1}{2} $-near-unanimity term
 which is not  $2n{-}4$-distributive. 
The proof relies heavily on \cite{mis}. 
We first need to establish a considerable amount
of notation and conventions.
In particular, we need to recast many notions
from \cite{mis} in the setting where
a dummy variable is added to a special set of lattice terms.

Throughout the present section
$n$ is a natural number $\geq 3$ and  $m=n+1$.
We first define the variety we shall use in order to provide the
main counterexample.

\begin{definition} \labbel{+lav}
Recall that  $m \geq 4$ and suppose that  $2 \leq j \leq m$.

(a) We define $u _{j,m}^+$
to be the $m{+}1$-ary 
lattice term  
\begin{equation}\labbel{+lte}
u _{j,m}^+(x_1, \dots, x_{m+1})= \prod _{|J|=j} \sum _{i \in J} x_i 
   \end{equation}    
where $J$ varies on  subsets of 
$\{ 1, 3,4, 5, \dots, m-1, m, m+1 \}$.
The definition is given modulo any fixed but otherwise arbitrary arrangement
of summands and factors; in fact, we shall never
use the actual term structure of $u _{j,m}^+$,
we shall only deal with the way it is evaluated.
The term $u _{j,m}^+$ corresponds to the term 
$u _{j,m}$ from \cite[Definition 3.3]{mis} when a dummy variable
is added at the second place.
This is necessary for definiteness, since we shall combine
lattice reducts endowed with operations given by
$u _{j,m}^+$ and Boolean reducts endowed with
an $m{+}1$-ary operation. 

(b) Again with $m \geq 4$ and  $2 \leq j \leq m$,
we let
$\mathbf N ^{j,m,+}$
be the $u _{j,m}^+$-reduct
of the two-element lattice with base set 
$\{ 0, 1\}$. 

(c) For $n \geq 3$, we let $\mathbf G^{{\rm nu}, n}$ 
denote the term reduct of the two-element Boolean algebra 
obtained by considering the $n{+}2$-ary term defined in \eqref{term},
 which we recall  below:
\begin{equation} \labbel{term2}   
u(x_1,  x_2, x_3, \dots , x_{n+2}) =  
  (x_1 + x_2') \prod _{\substack{1 \leq i< j \leq n+2 \\ i \neq 2, j \neq 2}} (x_i + x_j) .   
 \end{equation}

(d) Suppose that $n \geq 4$ and  $m=n+1$.
Let 
$\ell = \frac{m+1}{2} $
if $m$ is odd and
  $\ell = \frac{m}{2} $
if $m$ is even. We define
$\mathcal {N}_{n \frac{1}{2} }$
to be the variety 
generated by the algebras
\begin{equation*}
\mathbf G^{{\rm nu}, n}, \quad
 \mathbf N ^{3, m,+}, \quad 
\mathbf N ^{4, m,+}, \quad 
 \dots, \quad \mathbf N ^{ \ell, m,+}.
\end{equation*}

Notice that all the above algebras have
an $n{+}2$-ary operation, since $m=n+1$
and $u _{j,m}^+$ is $m{+}1$-ary.
Hence the definition is correct.  

(e) Under the assumptions in (d), we let 
$\mathcal {N}^{3,+}_{m}$
 be the variety 
generated by the algebras
\begin{equation*}
 \mathbf N ^{3, m,+}, \quad 
\mathbf N ^{4, m,+}, \quad 
 \dots, \quad \mathbf N ^{ \ell, m,+}.
\end{equation*}

\red{(f) In (d) we have defined 
$\mathcal {N}_{n \frac{1}{2} }$
for $n \geq 4$; however, we need to define
$\mathcal {N}_{n \frac{1}{2} }$ also 
for $n=3$. We let  
$\mathcal {N}_{3 \frac{1}{2} }$
be the variety generated by 
$\mathbf G^{{\rm nu}, 3}$.
The variety $\mathcal {N}_{3 \frac{1}{2} }$ 
is term-equivalent to a variety 
we have called $\mathcal I^-_4$
in \cite[Section 4, p. 15]{mis}.
We now recall the definition of $\mathcal I^-_4$
from \cite{mis}.} 
The variety $\mathcal I^-_4$ is
generated by term reducts of Boolean algebras
by considering both the terms $f(x,y,z)=x(y'+z)$
and $u _{2,4}(x_1,x_2,x_3,x_4) = \prod _{i < j \leq4}  (x_i + x_j)$. 
We now notice that these terms are expressible as a function of $u$
 from \eqref{term2} in the case $n=3$.
Indeed, $f(x,y,z)=u(z,y,x,x,x)$ and 
   $u _{2,4}(x_1,x_2,x_3,x_4)=u(x_1,x_1,x_2,x_3,x_4)$.
Conversely, $u$ can be expressed as 
$u(x_1,x_2,x_3,x_4,x_5) = 
f(u _{2,4}(x_1,x_3,x_4,x_5),x_2,x_1)$.
Hence $\mathcal {N}_{3 \frac{1}{2} }$ 
and  $\mathcal I^-_4$ are term-equivalent.
 \end{definition}   

\begin{remark} \labbel{op3}
If $m \geq 5$, then the operation in $\mathcal {N}^{3,+}_{m}$
has the property that, disregarding the second argument, if all 
the other arguments but two are given the same value, then
the operation returns this value. In fact, this property holds
in all the generating algebras, since the first upper indices 
in \ref{+lav}(e) are all $\geq 3$,
namely,  $j$ in \eqref{+lte} is always $\geq 3$.   
Moreover, since $m \geq 5$, we always have 
 $m \geq j+2$, for $j= 3,4, \dots, \ell$, because of the definition 
of $\ell$.  
  \end{remark}   

It is  convenient to extend the  $ {\stackrel{k}{\dots}}  $ notation
from the introduction. If $R$ and $S$ are reflexive binary relations,
we let $ R \circ S \circ \stackrel{1} {\dots} = R$
and  
$ R \circ S \circ \stackrel{0} {\dots} = 0$,
where $0$ is the minimal  
congruence on the algebra under consideration.
Moreover, $R^k$ is 
$ R \circ R \circ \stackrel{k} {\dots} $,
in particular, $R^0 = 0$.  

\begin{lemma} \labbel{maj}
Let $n \geq 4$ and
$m=n+1$.
Then there are  
an algebra $\mathbf A_3^+ \in \mathcal {N}_m^{3,+}$
and a subalgebra $\mathbf F$ of 
$ \mathbf A_3^+ \times 
\mathbf G^{{\rm nu}, n}$ 
such that the congruence identity
\begin{equation}\labbel{blah}
\tilde \alpha ( \tilde\beta \circ \tilde \gamma ) \subseteq 
(\tilde \alpha (\tilde \gamma \circ \tilde \beta )) ^{m-4} 
   \end{equation}    
fails in $\mathbf F$.

Moreover, the failure of 
\eqref{blah} 
can be witnessed 
by elements 
$(a,1)$, $(d,1)$, $(c,0)$ 
and congruences $ \tilde \alpha $,
\red{induced by }
$ 1 \times 0$, $\tilde \beta$, $\tilde \gamma$ 
of $\mathbf F$
such that   
$(a,1) \mathrel { \tilde \alpha  } (d,1)$,
$(a,1) \mathrel { \tilde \beta  } (c,0) \mathrel { \tilde \gamma  }  (d,1)$
and
$((a,1),(d,1)) \notin (\tilde \alpha ( \tilde \gamma \circ \tilde \beta )) ^{m-4} $.
 \end{lemma} 

\begin{proof}
This is proved like the Claim and the Subclaim
in the proof of Theorem 3.6 in \cite{mis}
with $j=3$ and $q=2$. 
The only difference is that here the second argument of the operations
need not be considered in the reasonings from \cite{mis}.
The reasonings obviously work here in the case of  
algebras in $\mathcal {N}^{3,+}_{m}$,
whose operation does not depend on the second argument
and is otherwise the same as in the variety $\mathcal {N}^{3}_{m}$
from \cite[p. 8]{mis}. On the other hand, the only property
of the other algebra \red{(called $\mathbf N^{2, m}$ there) }
used in the Claim in \cite[p. 8]{mis} is that 
if two $0$'s appear in the arguments of the operation, then the
outcome is $0$. Not considering the second argument,
this property  is true of
$\mathbf G^{{\rm nu}, n}$, as well, since its operation is defined by \eqref{term2}.   
\end{proof}

The next lemma is essentially a special case of 
\cite[Lemma 2.2]{mis} for terms with an additional variable.
We give a direct proof since it is simpler
than establishing the  notation
necessary for exploiting the connection with 
\cite[Lemma 2.2]{mis}.
\red{The ``types'' in the statement of the next lemma
are marked with a $\sigma$ in order to distinguish them 
from similar  types  
 we have used in some other places, permuting
the coordinates.
The issue is discussed in \cite[p. 4]{mis}.} 

\begin{lemma} \labbel{nua}
Suppose that $n \geq 4$ and $m=n+1$.  
Let $\mathbf 4$
denote the four-element Boolean algebra
with base set $\{ 0,e,e',1\}$ and 
let $\mathbf A$
be the $u$-term reduct of  $\mathbf 4$,
where $u$ is the term defined in equation  \eqref{term2}
in Definition \ref{+lav}(c). 

Let $\mathbf A_3^+ \in \mathcal {N}_m^{3,+}$,
 $\mathbf F \subseteq 
 \mathbf A_3^+ \times \mathbf G^{{\rm nu}, n} $,
$a, d \in A_3^+$ 
 and 
let $B$ be the subset of $  A \times A \times F $
consisting of those elements which have 
at least one of the following types
(modulo the natural identification of a triple
containing a subpair with a quadruple)
\begin{equation*}
\begin{gathered}
\text{Type I$^ \sigma$} \\
(\hyphe, 0, a, \hyphe)
\end{gathered} 
\qquad\qquad
\begin{gathered}
\text{Type II$^ \sigma$} \\
(0, 0, \hyphe, \hyphe),
  \end{gathered} 
\qquad\qquad
\begin{gathered}
\text{Type III$^ \sigma$} \\
(0, \hyphe, d, \hyphe)
  \end{gathered} 
\qquad\qquad
\begin{gathered}
\text{Type IV$^ \sigma$} \\
(\hyphe, \hyphe,  \hyphe, 0),
  \end{gathered} 
\end{equation*}    
where dashed places
can be filled with arbitrary elements from the 
corresponding algebras,
under the provision that each $4$-uple 
actually belongs to $  A \times A \times F $,
namely, that the pair consisting of the last two coordinates
belongs to  $F$.

 Then $B$ is the base set for a subalgebra
$\mathbf B$  of 
$\mathbf A \times \mathbf A \times \mathbf F$,
hence also a subalgebra  of $\mathbf A \times \mathbf A \times 
 \mathbf A_3^+ \times \mathbf G^{{\rm nu}, n}$.
 \end{lemma}

\begin{proof}
First, notice that $B$ is nonempty, for example,
\red{$(0,0,f_1,f_2) \in B$, if $(f_1,f_2)\in F$. }
Let $b_1, b_2, \dots, b_{n+2} \in B$.
We have to show that   
$b=u(b_1, b_2, \dots b_{n+2}) $, as computed in 
$\mathbf A \times \mathbf A \times 
 \mathbf A_3^+ \times \mathbf G^{{\rm nu}, n}$,
belongs to $  B$. 
Since the last two coordinates of 
each $b_i$ form a pair in $  F$,
then   the last two coordinates of 
 $b$ form a pair belonging to $ F$,
since $ \mathbf F$ is a subalgebra of 
$\mathbf A_3^+ \times \mathbf G^{{\rm nu}, n}$.
It remains to show that
$b$ has one of the types 
 I$^ \sigma$ - IV$^ \sigma$.

In the following considerations we shall deal with 
the $n{+}1$-element set 
$ \{ b_1, b_3, b_4, \allowbreak  \dots b_{n+2}\}$;
the element $b_2$ shall play no role in the discussion.
If two or more among the 
 $b_i$'s (not considering $b_2$) have $0$ at the fourth component, 
then   
 $u(b_1, b_2, \dots b_{n+2}) $ has 
$0$ at the fourth component, since the fourth component
is evaluated in 
$\mathbf G^{{\rm nu}, n}$.
Here we have used the fact, already
mentioned in the proof of Lemma \ref{maj},  that
$0$ in $\mathbf G^{{\rm nu}, n}$ is ``$2$-absorbing disregarding  the
second argument'' for  
 the term 
defined in  equation   \eqref{term2}. 
In this case $b$ has type IV$^ \sigma$, hence we are done.

Therefore we may suppose that 
at most one  $b_i$ has $0$ at the fourth component. 
If two or more among the 
 $b_i$'s  have $0$ at the first component
and 
two or more among the 
 $b_i$'s  have $0$ at the second component
 (again, in both cases, not considering $b_2$),
then $b$ has $0$
both at the first and at the second component,
again by the absorbing properties of $0$, 
thus $b$ has type    II$^ \sigma$, hence 
$b \in B$ and we are done in this case, too.

Otherwise, again not considering $b_2$,
 there are, say, at least $n-1$ many 
$b_i$'s with the first component different from $0$,
hence of type   I$^ \sigma$, since
we have excluded type  IV$^ \sigma$
and we are dealing with elements having at least one type
from I$^ \sigma$ - IV$^ \sigma$.
Such $b_i$'s have thus $a$ at the third component, 
hence $b$  has $a$ at the third component, by Remark \ref{op3}, to the effect that   
the operation on the third coordinate is ``$n{-}1$-majority disregarding the second
argument''.
 Thus $b$ has type  I$^ \sigma$, hence $b  $ belongs to $ B$.

Similarly, if there are at least $n-1$ many 
$b_i$'s with the second  component different from $0$,
then they are  of type   
\red{III$^ \sigma$ }
and then $b$, too, has type  
\red{III$^ \sigma$, }
thus $b \in B$.
\end{proof}

\begin{theorem} \labbel{thd}
For every $n \geq 3$, the variety 
$\mathcal {N}_{n \frac{1}{2} }$
introduced in Definition \ref{+lav}(d)(f) 
 has an $n\frac{1}{2} $-near-unanimity term
and  is not  $2n{-}4$-distributive.

More generally,  the congruence identity
\begin{equation}\labbel{blahh}
  \alpha (  \beta \circ   \gamma ) \subseteq 
 \alpha \beta  \circ (  \alpha (  \gamma \circ   \beta )) ^{n-3} 
\circ \alpha \gamma  
   \end{equation}    
fails in $\mathcal {N}_{n \frac{1}{2} }$.
 \end{theorem}

 \begin{proof}
\red{We first deal with the special case $n=3 $. 
 The first paragraph in the proof of 
Theorem \ref{ip}(4) shows that  $\mathcal {N}_{3 \frac{1}{2} }$ 
has a $3 \frac{1}{2} $-near-unanimity term.
On the other hand,  $\mathcal {N}_{3 \frac{1}{2} }$
is not $2$-distributive.
Indeed, in Definition   \ref{+lav}(f)
we have showed that $\mathcal {N}_{3 \frac{1}{2} }$ 
is term-equivalent to the variety 
 $\mathcal I^-_4$
from \cite[Section 4, p. 15]{mis}.
Moreover, \cite[Proposition 4.4]{mis}
shows that $\mathcal I^-_4$ is not 
$2$-distributive, hence 
$\mathcal {N}_{3 \frac{1}{2} }$ is not 
$2$-distributive, since, by definition, $2$-distributivity is preserved
under term-equivalence.
It follows that 
\eqref{blahh} fails for $\mathcal {N}_{3 \frac{1}{2} }$
when $n=3$, since, by
 our conventions,
if $n=3$, then \eqref{blahh}
reads $  \alpha (  \beta \circ   \gamma ) \subseteq 
 \alpha \beta  \circ  \alpha \gamma  $, and this identity is
equivalent to
$2$-distributivity, by the comment in Definition \ref{ndist}.    
We have proved the theorem in the case $n=3$.} 

So let $n \geq 4$. 
The variety $\mathcal {N}_{n \frac{1}{2} }$
has an $n\frac{1}{2} $-near-unanimity term
since the operation in each generating algebra is 
$n\frac{1}{2} $-near-unanimity.
Indeed, as already mentioned in the proof of
Theorem \ref{ip}(4), the term 
introduced in \eqref{term} and recalled in \eqref{term2}  
induces an $n\frac{1}{2} $-near-unanimity
operation.
Moreover, for $3 \leq j \leq \ell $, the operation on the algebras 
$\mathbf N ^{j,m,+}$ introduced in Definition \ref{+lav}(b) 
is $n\frac{1}{2} $-near-unanimity.
In fact, disregarding the second variable, the term $u _{j,m}^+$
from equation \eqref{+lte} 
is a near-unanimity term, hence equations \eqref{b2} hold.
Since  $u _{j,m}^+$ does not depend on the second variable, 
\eqref{b1} holds, as well.
Finally, \eqref{b3} holds in view of Remark \ref{op3} 
and since, again,  $u _{j,m}^+$ does not depend on the second variable.
 Notice that the assumption that $3 \leq j \leq \ell $
is used only  in order to get equation \eqref{b3}. 

We now show that the second statement
of the theorem  implies the failure of 
 $2n{-}4$-distributivity.
Indeed, since $ \alpha \gamma \circ \alpha \beta \subseteq 
 \alpha (  \gamma \circ   \beta )$, if equation \eqref{blahh}
fails, then also 
$   \alpha (  \beta \circ   \gamma ) \subseteq 
 \alpha \beta  \circ  \alpha   \gamma \circ  {\stackrel{2n-4}{\dots}}
 \circ   \alpha \gamma $ fails. \red{By the comment
in Definition \ref{ndist}, this means exactly that 
 $2n{-}4$-distributivity fails.} 

 It remains to show that 
 \eqref{blahh} fails
in $\mathcal {N}_{n \frac{1}{2} }$.
\red{Let $\mathbf A_3^+ \in \mathcal {N}_m^{3,+}$,
 $\mathbf F \subseteq 
 \mathbf A_3^+ \times \mathbf G^{{\rm nu}, n} $ and
$a, d \in A_3^+$ 
be given by Lemma \ref{maj}.
Correspondingly, }
let $\mathbf B$  be the algebra given by Lemma \ref{nua}.
Then $\mathbf  B \in \mathcal {N}_{n \frac{1}{2} }$, since 
$\mathbf  B$ is a subalgebra of 
$\mathbf A \times \mathbf A \times \mathbf A_3^+ \times 
\mathbf G^{{\rm nu}, n}$,
$\mathbf A_3^+ \in \mathcal {N}_m^{3,+}$,
$\mathcal {N}_m^{3,+}$ is a subvariety of 
 $\mathcal {N}_{n \frac{1}{2} }$,
$\mathbf G^{{\rm nu}, n} \in \mathcal {N}_{n \frac{1}{2} }$ by construction  and 
$\mathbf A$ belongs to the variety generated by 
$\mathbf G^{{\rm nu}, n}$, hence 
$\mathbf A \in \mathcal {N}_{n \frac{1}{2} }$. 

Define the following congruences on $\mathbf A$.
The congruence $\beta^*$
is the congruence induced by the partition
$\{ \{1,e  \} ,\{e', 0  \} \}$
and   
$ \gamma ^*$
is the congruence induced by the partition
$\{ \{1,e'  \} ,\{e, 0  \} \}$.
Since each of the above partitions induces a congruence on the four-element
Boolean algebra, then each partition induces a congruence on any term reduct.

\red{Since $\mathbf  B$ is a subalgebra of  $\mathbf A \times \mathbf A \times 
\mathbf F $, then every congruence
on $\mathbf A \times \mathbf A \times 
\mathbf F $ induces a congruence on $\mathbf  B$.} 
Let $\beta$ and $\gamma$ be, respectively,  the congruences
induced on $\mathbf  B$ by the congruences
$\beta^* \times \beta ^* \times \tilde \beta $
and 
$ \gamma ^* \times \gamma  ^* \times \tilde \gamma $
on $\mathbf A \times \mathbf A \times 
\mathbf F $, where $ \tilde \alpha $, $ \tilde \beta $ and $ \tilde \gamma  $
are given by   Lemma \ref{maj}.
Similarly, let $\alpha$ be induced by
$1 \times 1 \times \tilde \alpha $.
Here, as in Lemma \ref{nua},
we are identifying a triple containing a subpair
 with a quadruple.

Let $a$, $d$ and $c$  be given by Lemma \ref{maj}  
and consider the elements
$ \bar a =(1,0,a,1)$
and $\bar d = (0,1,d,1)$ in $B$
of types, respectively, I$^ \sigma$ and III$^ \sigma$.
Clearly, $ \bar{a} \mathrel { \alpha  } \bar{d}  $, recalling
from Lemma \ref{maj} that 
$(a,1) \mathrel {\tilde \alpha} (d,1) $. 
The element 
$ \bar c = (e,e', c,0)$ of type IV$^ \sigma$
witnesses that  
 $( \bar{a}, \bar{d}) \in \alpha ( \beta \circ \gamma )   $,
since  
$ \bar{a} \mathrel { \beta   } \bar{c} \mathrel { \gamma  }   \bar{d}  $.
We are going to show that 
 $( \bar{a}, \bar{d}) \notin  \alpha \beta
  \circ (  \alpha (  \gamma \circ   \beta )) ^{n-3} 
\circ \alpha \gamma   $, hence \eqref{blahh}
fails in an algebra in $\mathcal {N}_{n \frac{1}{2} }$. 

Suppose by contradiction that 
 $( \bar{a}, \bar{d}) \in  \alpha \beta
  \circ (  \alpha (  \gamma \circ   \beta )) ^{n-3} 
\circ \alpha \gamma   $,
hence there are elements 
$ \bar{g}, \bar{h} \in B$ such that
  $ \bar{a} \mathrel { \alpha  \beta } \bar{ g} $,
$( \bar{g}, \bar{h}) \in  (  \alpha (  \gamma \circ   \beta )) ^{n-3}$
and 
 $ \bar{ h} \mathrel { \alpha  \gamma  }  \bar{d}  $.
By $ \beta $-equivalence,
the first component of $ \bar{g}$ 
is either $1$ or $e$, in any case, not $0$.
By $\alpha$-equivalence, the last component of 
$\bar{g}$ is $1$, \red{since
$ \tilde \alpha $ is induced by 
$1 \times 0$ on 
$\mathbf F \subseteq  \mathbf A_3^+ \times 
\mathbf G^{{\rm nu}, n}$, by Lemma \ref{maj}. }
Since $\bar{g}$ must be in $B$, then
$\bar{g}$ has type   I$^ \sigma$,
hence its third component is $a$. 
Symmetrically, 
\red{the second component of 
$\bar{h}$ is not $0$, }
the last component of 
$\bar{h}$ is $1$,
$\bar{h}$ has type   III$^ \sigma$
and its third component is $d$. 
From $( \bar{g}, \bar{h}) \in  (  \alpha (  \gamma \circ   \beta )) ^{n-3}$
we then get
$( (a,1), (d,1)) \in  ( \tilde \alpha (  \tilde \gamma \circ   \tilde \beta )) ^{n-3}$,
contradicting Lemma \ref{maj} (recall that $m=n+1$). 
 \end{proof}

\begin{remark} \labbel{altr}
As a corollary  of Theorem \ref{thd},
we get another proof of 
Theorem  \ref{ip}(4).
Indeed,  $\mathcal {N}_{n \frac{1}{2} }$
is locally finite, being the join of locally finite varieties,
$\mathcal {N}_{n \frac{1}{2} }$ 
  has an $n\frac{1}{2} $-near-unanimity term,
but 
$\mathcal {N}_{n \frac{1}{2} }$ 
  has not an $n $-near-unanimity term,
otherwise it would be
  $2n{-}4$-distributive, by Theorem \ref{mit}. 
This contradicts Theorem \ref{thd}.
 \end{remark}

\section {Modularity levels and more identities} \labbel{mod}

\begin{proposition} \labbel{ippm}
If $n \geq 3$, then every variety $\mathcal V$ with an
$n\frac{1}{2} $-near-unanimity term
 is $2n{-}2$-modular. 
 \end{proposition}  

\begin{proof}
If $u$ is an $n\frac{1}{2} $-near-unanimity term,  define
\begin{equation}\labbel{lungh}      
\begin{aligned}
t_0(x,y,w,z) &= x,
\\
t_1(x,y,w,z) &= u(x,x, x, x,x, x, \dots,x, x,x, y, z),
\\
t_2(x,y,w,z) &= u(x,x, x, x, x, x,\dots, x, x,x, w, z),
\\
t_3(x,y,w,z) &= u(x,x, x, x, x, x,\dots, x, x,y, z, z),
\\
t_4(x,y,w,z) &= u(x,x, x, x, x, x,\dots, x, x,w, z, z),
\\
t_5(x,y,w,z) &= u(x,x, x, x,x, x, \dots, x, y,z, z, z),
\\
&\dots
\\
t_{2n-7}(x,y,w,z) &= u(x,x, x, x,y, z, \dots, z, z,z, z, z), 
\\
t_{2n-6}(x,y,w,z) &= u(x,x, x, x,w, z, \dots, z, z,z, z, z), 
\\
t_{2n-5}(x,y,w,z) &= u(x,x, x, y,z, z, \dots, z, z,z, z, z), 
\\
t_{2n-4}(x,y,w,z) &= u(x,x, x, w,z, z, \dots, z, z,z, z, z), 
\\
t_{2n-3}(x,y,w,z) &= u(y,w, z, z,z, z, \dots, z, z,z, z, z), 
\\
t_{2n-2}(x,y,w,z)&= z.
 \end{aligned}
\end{equation} 
\red{The terms $t_{0}, \dots, t_{2n-2}$ 
satisfy the equations in Theorem \ref{daythm}(2), hence $\mathcal V$
is $2n{-}2$-modular.} 
 \end{proof} 

Except for the penultimate line, the 
construction of the terms in \eqref{lungh}  is identical 
with \cite[Theorem 3.19]{S}. From another point of view, 
the proof of  Proposition \ref{ippm} exploits the fact
mentioned in Remark \ref{mki}(b) 
that a variety with an  $n \frac{1}{2} $-near-unanimity term 
has directed Gumm terms;  then classical arguments from
\cite{D,G,LTT} can be used in order to get
Day  terms  from a sequence of ternary terms. 

\red{Recall that item (1) in Theorem \ref{tut} has been proved 
at the end of Section \ref{bsec}. }
We can now complete the proof.

 \begin{proof}[Proof of Theorem \ref{tut} (continued)]
(2) When $h$ is an integer, the first statement
in (2) is Mitschke's Theorem \cite[Theorem 2]{Mi},
reported here in Theorem \ref{mit}.
When $h$ is a half-integer,
the first statement in (2) is
Proposition \ref{ipp}. 
Indeed, in the latter case $n= h - \frac{1}{2} $, hence 
$2n-3=2h-4$.    

If $h$ is an integer, the last statement in (2) 
is Theorem 3.6(1) in \cite{mis}.    
If $h$ is a half-integer,  
then a counterexample is the variety  
$\mathcal {N}_{n \frac{1}{2} }$ introduced in Definition \ref{+lav}(d)(f),
as shown in Theorem \ref{thd}.

(3) 
When $h$ is an integer the first part
follows from  \cite[Theorem 3.19]{S}. When $h$ is a half-integer it
follows from   
Proposition \ref{ippm}.

As for the second part, if $h$ is an integer,
this follows from  \cite[Theorem 3.6(3)]{mis}.
\qedhere$_{\text{ to be continued}}$\end{proof}
 
\red{In order to deal with the remaining case,
 we need the $3$-dimensional
analogue of  Lemma \ref{maj}.

\begin{lemma} \labbel{maj3}
Let $n \geq 4$ and
$m=n+1$.
Then there are  
an algebra $\mathbf A_3^+ \in \mathcal {N}_m^{3,+}$
and a subalgebra $\mathbf F$ of 
$ \mathbf A_3^+ \times 
\mathbf G^{{\rm nu}, n}$ 
such that the congruence identity
\begin{equation}\labbel{modblah}
\tilde \alpha ( \tilde\beta \circ \tilde \alpha \tilde \gamma \circ \tilde\beta )
 \subseteq 
(\tilde \alpha (\tilde \gamma \circ  \tilde \beta \circ \tilde \gamma) )^{m-4} 
   \end{equation}    
fails in $\mathbf F$.

Moreover, the failure of 
\eqref{modblah} 
can be witnessed 
by elements 
$(a,1)$, $(d,1)$, $(c_1,0)$, $(c_2,0)$ 
and congruences $ \tilde \alpha $,
induced by 
$ 1 \times 0$, $\tilde \beta$, $\tilde \gamma$ 
of $\mathbf F$
such that   
$(a,1) \mathrel { \tilde \alpha  } (d,1)$,
$(a,1) \mathrel { \tilde \beta  } (c_1,0) 
\mathrel { \tilde \alpha \tilde \gamma  } (c_2,0) 
\mathrel { \tilde \beta  }  (d,1)$
and
$((a,1),(d,1)) \notin 
(\tilde \alpha (\tilde \gamma \circ  \tilde \beta \circ \tilde \gamma) )^{m-4}  $.
 \end{lemma}

\begin{proof}
This corresponds to the case 
$j=3$ and $q=3$ in the Claim in \cite{mis}
and is proved
using the remarks in the proof of Lemma \ref{maj} here.    
Notice that we have not indicated the dependency
on $q$ in \cite{mis}, thus the algebra
 $\mathbf A_3^+$ here is not the same algebra as in Lemma \ref{maj}.
\end{proof}}

 \begin{proof}[Proof of Theorem \ref{tut} (continued)]
\red{The remaining case to be proved
is the last statement in (3) when 
$h >4$ is a half integer. 
So let $h=n  \frac{1}{2} $, $n \geq 4$
and $m=n+1=h+ \frac{1}{2}$. 
Consider the algebras and elements given by
Lemma \ref{maj3}.
Since
\eqref{modblah} 
fails in $\mathbf F$, then
\eqref{modblah} fails when we consider 
 $\tilde \alpha \tilde \gamma$ in place of $\tilde \gamma$.
Since   $\tilde \alpha (\tilde \alpha \tilde \gamma \circ  
\tilde \beta \circ \tilde \alpha \tilde \gamma) =
\tilde \alpha \tilde \gamma \circ  \tilde \alpha
\tilde \beta \circ \tilde \alpha \tilde \gamma$,
we get that 
\begin{equation}\labbel{modblahk}
\tilde \alpha ( \tilde\beta \circ \tilde \alpha \tilde \gamma \circ \tilde\beta )
 \subseteq 
\tilde \alpha \tilde \gamma \circ \tilde \alpha \tilde \beta  
\circ  {\stackrel{2m-7}{\dots}} \circ \tilde \alpha \tilde \gamma   
   \end{equation}    
fails in $\mathbf F$. We have used the fact that
$m-5$ pairs of factors of the form $\tilde \alpha \tilde \gamma$ mutually absorb,
when computing $(\tilde \alpha \tilde \gamma \circ  \tilde \alpha
\tilde \beta \circ \tilde \alpha \tilde \gamma)^{m-4}$. 

Now the proof of 
 Theorem \ref{thd} carries over with no essential
 modification
in order to show that
\begin{align}\labbel{modblahh}
  \alpha (  \beta \circ   \alpha   \gamma \circ  \beta )
& \subseteq 
\alpha \beta \circ (  \alpha   \gamma \circ   \alpha   \beta  
\circ  {\stackrel{2m- 7}{\dots}} \circ   \alpha   \gamma )
\circ \alpha \beta, \quad \text{that is,}
\\ 
\labbel{modblahhh}
  \alpha (  \beta \circ   \alpha   \gamma \circ  \beta )
 &\subseteq 
\alpha \beta \circ  \alpha   \gamma   
\circ  {\stackrel{2m- 5}{\dots}} 
\circ \alpha \beta
   \end{align}    
 fail in $\mathcal {N}_{n \frac{1}{2} }$.

In more detail, 
let $\mathbf A_3^+ \in \mathcal {N}_m^{3,+}$,
 $\mathbf F \subseteq 
 \mathbf A_3^+ \times \mathbf G^{{\rm nu}, n} $ and
$a, d \in A_3^+$ 
be given by Lemma \ref{maj3} and
let $\mathbf B$  be the
subalgebra  of $\mathbf A \times \mathbf A \times 
 \mathbf A_3^+ \times \mathbf G^{{\rm nu}, n}$
  given by Lemma \ref{nua}.
Let $\alpha$, $\beta$, $\gamma$, 
$ \bar a =(1,0,a,1)$
and $\bar d = (0,1,d,1)$
 be defined as in the
proof of \ref{thd}. The elements $\bar c_1=(e,0,c_1,0)$ and 
$\bar c_2=(0,e,c_2,0)$
witness that $(\bar a, \bar d) \in \alpha (\beta \circ \alpha \gamma \circ \beta) $.
We claim that, on the other hand, 
$(\bar a, \bar d)$ does not belong to 
the right-hand side of \eqref{modblahh}.
In comparison with the proof of Theorem \ref{thd},
 here we need to add a factor of the form
 $\alpha \beta $, rather than $\alpha \gamma $, at the 
right outer edge of \eqref{modblahk},
as in \eqref{modblahh}, compare \eqref{blahh}. 
This involves no change in 
the proof. 
Indeed, since 
$\bar d = (0,1,d,1)$, if
 $ \bar{ h} \mathrel { \alpha  \beta   }  \bar{d}  $,
then 
the second component of $ \bar{h}$ 
is distinct from $0$ 
by $ \beta $-equivalence.
The last component of $ \bar{h}$ 
is $1$ 
by $\alpha$-equivalence, hence 
$ \bar{h}$  has type 
   III$^ \sigma$,
thus it has  $d$ 
as the  third component.
All the rest is identical to the proof of \ref{thd}.

Since   $m=h+ \frac{1}{2} $, then  
$2m- 5 =2h-4$, hence the failure of \eqref{modblahhh}
shows that $2h{-}4$-modularity fails in  
$\mathcal {N}_{n \frac{1}{2}}$, in virtue of the comment after
Theorem \ref{daythm}.}
\end{proof}

\begin{remark} \labbel{more}
The above arguments together with the proof of 
Theorem 3.6(4) in \cite{mis} show that if $n \geq 4$,
then   the following congruence identity fails in $\mathcal {N}_{n \frac{1}{2} }$
\begin{equation}\labbel{blbl}
\alpha (\beta \circ ( \alpha \gamma \circ \alpha \beta 
\circ {\stackrel{q-2}{\dots}} \circ \alpha \beta ^\bullet   )
\circ   \gamma ^ \bullet) 
\subseteq
\alpha \beta \circ 
(\alpha ( \gamma \circ \beta \circ {\stackrel{q}{\dots}}
 \circ  \beta ^ \bullet  )) ^{n- \red{3}
}
\circ \alpha \gamma ^ \bullet ,  
   \end{equation}    
for every $q \geq 2$,
 where 
$\beta^ \bullet = \beta $,
$ \gamma ^ \bullet = \gamma  $
if $q$ is even and 
$\beta^ \bullet = \gamma $,
$ \gamma ^ \bullet = \beta   $
if $q$ is odd.
 \end{remark}

\section{Some variations} \labbel{var} 

\begin{remark} \labbel{dueemm}
The notion of an $n\frac{1}{2} $-near-unanimity term
makes sense for $n=2$; in this case we get a
 $\frac{5}{2} $-near-unanimity term. Such   a term is required to satisfy
\begin{equation}    
\begin{aligned} \labbel{dueemmeq} 
u(z,z, x, x) &= x,  & u(x,z, x, x) &= x, 
& u(x, x,z, x) &= x, 
\\
  u(x, x, x,z) &= x,  
& u(x, z, z, z) &= x,
 \end{aligned}   
  \end{equation}
where the last equation follows from the fact that in the case $n=2$
\eqref{b3} reads $u(x, x, x, z) = u(x, z, z, z)$.

We shall show that the existence of a $\frac{5}{2} $-near-unanimity term
is equivalent to the existence of a Pixley term.
Indeed, if we let $t(x,y,z)= u(x,y,z,z)$, then 
$t$ is a Pixley term witnessing arithmeticity.
\red{Notice that we do not need the third and fourth equations
in \eqref{dueemmeq}. }
Conversely, if $t$ is a Pixley term, then
$u(x,y,z,w)= t(x,t(y,z,w), w)$    is a 
 $\frac{5}{2} $-near-unanimity term.
In this respect, the proof of Theorem \ref{ip}(1)
generalizes the fact that if  $t$ is a Pixley term, then
$s(x,z,w)= t(x,t(x,z,w), w)$ is a majority term.

The existence of a $2$-near-unanimity term 
might be interpreted as a condition implying that we are in a trivial 
variety; in this sense, a  $\frac{5}{2} $-near-unanimity term
is a condition strictly between a $2$-near-unanimity term
and a $3$-near-unanimity term, extending Theorem \ref{ip}. 
On the other hand, a  $\frac{5}{2} $-near-unanimity term,
equivalently, a Pixley term, does not imply
$2n-4=1$-distributivity (= being a trivial variety), hence the assumption
$n \geq 3$ is necessary in Proposition \ref{ipp}.
 \end{remark}

\begin{remark} \labbel{piuid} 
The term $u$ defined by equation
\eqref{term} for $n \geq 3$  in the proof of Theorem \ref{ip}(4)
satisfies many more equations in Boolean algebras, besides the equations 
\eqref{b1}-\eqref{b3}  defining an 
$n\frac{1}{2} $-near-unanimity term.
For example, $u$ satisfies 
\begin{align} 
\labbel{uu1}
u(x,z,x, \dots, x,\underset{i}{y},x, \dots, x) &= x, 
\\
\labbel{uu2}   
u(x,y,x_3,x_4, \dots,\underset{i}{x}, \dots, x_{n+2}) &=
u(x,z,x_3,x_4, \dots, \underset{i}{x}, \dots, x_{n+2}),
\\
\labbel{uu3}
u(x,x,z,z, \dots, z,\underset{i}{x},z, \dots, z)& = 
 u(x,z,z,z, \dots, z,z,z, \dots, z),
\\
\labbel{uu4}   
u(x_1,x_2,x_3,x_4, \dots, x_{n+2})&=
u(x_1,x_2,x_{ \tau (3)}, x_{ \tau (4)},\dots, x_{ \tau (n+2)}),
 \end{align}   
where in \eqref{uu1} - \eqref{uu3} $i$ varies with  $3 \leq i \leq n+2$
and in \eqref{uu4}  
 $\tau$ is a permutation of the set $\{ 3, 4, \dots, n+2\}$.
\end{remark}

\begin{corollary} \labbel{inpiuid}
Theorems \ref{tut} and  \ref{ip} 
  hold if in the definition of an $n\frac{1}{2} $-near-unanimity term 
we add some or all of the equations \eqref{uu1} - \eqref{uu3}.

Theorem \ref{thd}
  holds if in the definition of an $n\frac{1}{2} $-near-unanimity term 
we add some or all of the equations \eqref{uu1} - \eqref{uu4}.

If some variety $\mathcal V$ has a 
symmetric $n$-near-unanimity term,
then $\mathcal V$ has  
an $n\frac{1}{2} $-near-unanimity term 
satisfying \eqref{uu1} - \eqref{uu4}.
 \end{corollary} 

\begin{proof} 
\red{Since Clauses (1) and (2) in Theorem \ref{ip} hold
for $n\frac{1}{2} $-near-unanimity terms,
these clauses still hold if we replace ``$n\frac{1}{2} $-near-unanimity term'' with
a stronger notion.
The argument in the proof of \ref{ip}(3)
carries over for \eqref{uu1} - \eqref{uu3}, since 
the first two variables are dummy.
The term $u$ defined by equation \eqref{term}
satisfies \eqref{uu1} - \eqref{uu4}, hence the proof
of \ref{ip}(4) carries over, too.  
Thus Theorem \ref{ip} 
  holds for the stronger notions when
we add some  equations from \eqref{uu1} - \eqref{uu3}.

In order to prove that 
Theorem \ref{thd}
  holds in the generalized setting, we just check 
that \eqref{uu1} - \eqref{uu4} are satisfied by the operation of
 $\mathcal {N}_{n \frac{1}{2} }$.  
It suffices to show that this holds for all the algebras generating 
$\mathcal {N}_{n \frac{1}{2} }$. }
We have just mentioned that the operation in  the algebra $\mathbf G^{{\rm nu}, n}$
introduced in Definition \ref{+lav}(c)  satisfies 
 \eqref{uu1} - \eqref{uu4}.
The term $u _{j,m}^+$ introduced in Definition \ref{+lav} 
does not depend on the second variable, hence satisfies 
 \eqref{uu2}. Moreover, $u _{j,m}^+$ 
 is near-unanimity 
with regard to the remaining variables, hence 
it satisfies \eqref{uu1}. 
It satisfies 
 \eqref{uu4} 
since it is symmetric, disregarding the second variable.
If $m \geq 5 $ and $m \geq j+2$, then 
$u _{j,m}^+$ satisfies \eqref{uu3}, since both members 
evaluate as $z$, by Remark \ref{op3}.   
In conclusion, 
all the algebras generating 
$\mathcal {N}_{n \frac{1}{2} }$ satisfy
\eqref{uu1} - \eqref{uu4}, hence such equations hold in
 $\mathcal {N}_{n \frac{1}{2} }$.  \red{This proves
the generalized version of Theorem \ref{thd}, then the general 
version of Theorem \ref{tut} follows, noticing that
in the final part of the proof of \ref{tut}(3) we use again
the variety $\mathcal {N}_{n \frac{1}{2}}$.}

Finally, if $w$ is a symmetric $n$-near-unanimity term,
then, adding two dummy variables at the first two places,   
equation \eqref{uu4} holds. Equations \eqref{uu1} - \eqref{uu3}
have been already taken care of. Thus  the last statement holds. 
\end{proof}  

In particular, the equations \eqref{b1} - \eqref{b3}, \eqref{uu1} - \eqref{uu4} 
together neither imply
$2n{-}4$-distributivity, nor 
imply the existence of an 
$n$-near-unanimity term.

\section{Some problems}\labbel{prob}  

We are not claiming that the  problems below are difficult.

\begin{problem} \labbel{wnum}
\red{Study the  notion of a 
\emph{weak $n\frac{1}{2} $-near-unanimity term},
that is, an idempotent $n{+}2$-ary  term satisfying
 \eqref{b3} from Definition \ref{12}, as well as
\begin{align*} 
u(z,z, x, x, \dots, x) &= u(x, \dots, x,\underset{i}{z},x, \dots, x),
 && \text{for $2 \leq i \leq n+2$.} 
 \end{align*}
Compare \cite{MM}.
One might possibly add some equations of the form
\eqref{uu2} - \eqref{uu4} and
\begin{align*} 
u(z,z, x, x, \dots, x) &= u(x, z, x \dots, x,\underset{i}{y},x, \dots, x),
 && \text{for $3 \leq i \leq n+2$} 
 \end{align*}
(compare \eqref{uu1}).}
 \end{problem}

Edge terms are an important generalization
of near-unanimity terms and, possibly
in equivalent formulations, play a chief role
in many computational problems, e.~g. 
\cite{AM,AMM,Bar2,BIM,BMS,IMM,KZ,KKS,KS,KS2} and further
references in the quoted papers.
A variety with an edge term is congruence modular \cite[Theorem 4.2]{BIM};
moreover, for $k \geq 3$, a congruence distributive variety $\mathcal V$  
has a $k$-edge term if and only if  
$\mathcal V$ has a $k$-near-unanimity term \cite[Theorem 4.4]{BIM}.
Henceforth the following problem suggests itself naturally.

\begin{problem} \labbel{edge} 
Is there a notion strictly  between a
$k$- and a $k{+}1$-edge term?

A variety with a 
$k$-edge term is $2k{-}3$-modular: 
\red{just forget about the last term
in the proof of \cite[Theorem 4.2]{BIM}, or \cite[Proposition 5.3]{B}.}
Hence a candidate for a ``$k \frac{1}{2} $-edge term''  
should imply 
 $2k{-}2$-modularity but not
 $2k{-}3$-modularity.
 \end{problem}

In  connection with Problem \ref{edge}, it is probably
interesting to study the following notion weaker than 
an $n \frac{1}{2} $-near-unanimity term. 

\begin{definition} \labbel{pseudo}
For $n  \geq 1$, a  \emph{skew-edge term} is an
$n{+}2$-ary term 
satisfying equations \eqref{b1}, \eqref{b3}
from Definition \ref{12}, as well as
equations \eqref{b2} for $4 \leq i \leq n+2$.
 \end{definition}   

An equivalent characterization of varieties 
with a  skew-edge term is given in Remark 
\ref{altrr} below. 
A skew-edge term implies congruence modularity,
but does not imply congruence distributivity.

\begin{proposition} \labbel{pse}
If $n \geq 3$, then every variety
with an $n{+}2$-ary  skew-edge term
  is $2n{-}2$-modular. 

Every congruence permutable variety 
has a skew-edge term.
 \end{proposition} 

 \begin{proof}
Only the equations satisfied by a  skew-edge term
are used in the proof of Proposition \ref{ippm}.
Compare also Remark \ref{mki}(a)(b).  

A $3$-ary term is a Maltsev term if and only if
it is a skew-edge term. In any case, adding dummy 
variables after the third, we get a    skew-edge term
of arbitrary arity.
 \end{proof}

\begin{problem} \labbel{3emmod}
Is there a variety $\mathcal V$ with a $3 \frac{1}{2} $-near-unanimity term 
and such that $\mathcal V$ is not $3$-modular?  
 \end{problem}  

\begin{problem} \labbel{menovar}
Is there a notion equivalent
(for varieties) to the existence of an 
$n\frac{1}{2} $-near-unanimity term
and whose definition  involves a term (or, anyway, a set
of terms)
of smaller arity?
(that is, of arity $<n+2$)

Is there a notion 
strictly between
an $n$-near-unanimity term  
and an $n{+}1$-near-unanimity term,  
 satisfying the analogue of Theorem \ref{tut} 
and whose definition  involves a single  term 
of  arity $n+1$?
 \end{problem}  

\begin{remark} \labbel{cort}
For $n \geq 3$, the condition 

($\diamondsuit_n$) there is an $n{+}1$-ary
near-unanimity term + $2n{-}3$-distributivity

\noindent
involves only terms of arity $\leq n+1$.
\red{The condition $\diamondsuit_n$ follows from the existence 
of an $n \frac{1}{2} $-near-unanimity term, by 
Proposition \ref{ipp} and Theorem \ref{ip}(1).
Thus, by Theorem \ref{ip}(4),  $\diamondsuit_n$ }
does not imply the existence of an $n$-near-unanimity term.
The condition $\diamondsuit_n$ does not imply
$2n{-}4$-distributivity, by 
Theorem \ref{thd}. 

Is there a $2n{-}3$-distributive variety 
with an $n{+}1$-near-unanimity term 
but without an $n \frac{1}{2} $-near-unanimity term? 

A $3$-distributive variety 
with an $n{+}1$-near-unanimity term 
but without an $n$-near-unanimity term
has been constructed in \cite[Proposition 4.4]{mis}.
\end{remark}

\begin{remark} \labbel{altrr}
For $n \geq 3$ and $\mathcal V$ a variety, 
the following conditions are equivalent. 
\begin{enumerate}[(i)]
   \item   
$\mathcal V$ has an $n{+}2$-ary term $u$  satisfying
\eqref{b1}, \eqref{b3} and \eqref{b2} for 
[$i=2$ and] $4 \leq i \leq n+2$,     
\item
$\mathcal V$ has 
an $n$-ary term $v$ and a ternary term $t$ such that 
\begin{equation}\labbel{altrreq}
\begin{aligned} 
&v(x, \dots, x,\underset{i}{z},x, \dots, x) = x, \qquad \text{for $2 \leq i \leq n$,}
\\ 
& v(x,z,z,z, \dots, z,z)= t(x,z,z)
\\
 & t(x,x,z)= z, \qquad [t(x,y,x)= x]
\end{aligned}
    \end{equation}      
 \end{enumerate} 

Indeed, if $u$ satisfies (i),
then $v(x_1, x_2, x_3, \dots)= u(x_1,x_1,x_1, \allowbreak  x_2, x_3, \dots)$
and  $t(x,y,z)= u(x,y,z,z,z, \dots )$ satisfy \eqref{altrreq}.
Conversely, if $v$ and $t$ are given by \eqref{altrreq}, then  
$u(x_1, x_2, x_3, x_4,\dots)= t(x_1,x_2,v(x_3, x_4,\dots))$ 
satisfies (i).

As we mentioned in Remark \ref{mki}(a), the case 
$i=3$ in equation \eqref{b2} is not necessary in order to 
get $2n{-}3$-distributivity.
As we mentioned in the proof of 
Proposition \ref{pse}, the cases 
$i=2 $ and  $i=3$ in  \eqref{b2} are not necessary in order to 
get $2n{-}2$-modularity.
However, the case $i=3$ seems necessary in order to get
that the existence of an $n \frac{1}{2} $-near-unanimity term 
implies the existence of  an $n{+}1$-near-unanimity term.   
\end{remark}

In \cite{CCV}
Campanella, Conley,  Valeriote 
proved that if $\mathcal V$ and $\mathcal W$ 
are idempotent varieties of the same type and with, respectively, 
an $n$-near-unanimity term
and an $m$-near-unanimity term,
then both the join and the Maltsev product \cite{FMMa}  of
$\mathcal V$ and $\mathcal W$ have an
$n{+}m{-}1$-near-unanimity term.
Moreover, they show by a counterexample that the result is
the best possible. 

\begin{problem} \labbel{campan}
Do the results from \cite{CCV}
hold also when one or both $n$ and  $m$ 
  are half-integers? Possibly, one needs to modify the definition
of a $n \frac{1}{2} $-near-unanimity term as in 
Remark \ref{piuid} and Corollary \ref{inpiuid}.  
 \end{problem}   

\begin{remark} \labbel{notsym}   
\red{Let $n \geq 3$ and $\mathcal V$ be the variety with a single
$n$-ary operation $u$ satisfying the near-unanimity equations 
and no other equation (except, of course, for those equations
which logically follow from the near-unanimity rule).
We claim that $\mathcal V$ has no symmetric near-unanimity term,
actually, no symmetric term of arity $\geq 2$. 

Indeed, every  term in $\mathcal {V}$ has a normal form, obtained
by applying the near-unanimity rule whenever possible
(in particular, this shows that $\mathcal V$ is not locally finite).
Suppose by contradiction that $\mathcal V$ 
has an $m$-ary  symmetric term
$s(x_1, \dots , x_m)$ for some $m \geq 2$ 
and choose such an $s$  in normal form
 of minimal complexity.
Since $s$ is symmetric and $m \geq 2$,
 then $s$ cannot be a variable,
hence $s$  is  written  as
\begin{equation*}   
s(x_1, \dots , x_m) = u(t_1(x_1, \dots , x_m), \dots, t_n(x_1, \dots , x_m) ),
 \end{equation*} 
for certain terms $t_1, t_2, \dots, t_n$.
Since $s$ is symmetric,  we  have
\begin{equation*}   
u(t_1(x_1, \dots , x_m), \dots) = s(x_1, \dots , x_m) = 
s(x_ { \sigma 1}, \dots , x_{ \sigma m}) =
u(t_1(x_ { \sigma 1}, \dots , x_ { \sigma m}), \dots )
\end{equation*}   
for every permutation $\sigma$ of $\{1, 2, \dots, m\}$, 
hence 
\begin{equation*} 
t_1(x_1, \dots , x_m) = t_1(x_ { \sigma 1}, \dots , x_ { \sigma m}),
\end{equation*}    
for every $\sigma$,  since we are dealing with normal forms.
Thus $t_1$ is symmetric of complexity less
than $s$, a contradiction.} 
 \end{remark}   

\begin{problem} \labbel{lfsym}
Prove or disprove.
A locally finite variety with a  near-unanimity term 
has a symmetric near-unanimity term.
 \end{problem}  

In this connection we just point out
that if $\mathcal V$ is a variety
(not necessarily locally finite)
with an $n$-ary near-unanimity term $u$ 
and with  a symmetric $n!$-ary
idempotent  term
$t$, then $\mathcal V$ has a symmetric
$n$-ary near-unanimity term $v$.
Just let 
$v(x_1, \dots , x_n)=t(\dots, u (x_{ \tau (1)}, \dots , x_{ \tau (n)}), \dots)$,
where $\tau$ varies among all permutations of
$\{ 1, \dots, n\}$ and  ``different  arguments of $t$
are filled using  different permutations''.

\emph{Acknowledgement.}
\red{We thank an anonymous referee for many useful suggestions
and for detecting a significant issue in a former version of the paper.

We thank K. Kearnes for many useful comments which
have been of great help in improving the manuscript.} 

\newcommand{\eeeemph}{{}}

\end{document}